\documentclass[final,1p,times]{elsarticle}

\usepackage{amsmath,amssymb,amsfonts,amsthm}

\usepackage{graphicx}

\usepackage{mathrsfs}
\usepackage{mathabx}
\usepackage{dsfont}

\usepackage{color}
\usepackage{cases}

\newtheorem{assumption}{Assumption}
\newtheorem{remark}{Remark}
\newtheorem{theorem}{Theorem}
\newtheorem{lemma}{Lemma}

\let\b=\boldsymbol

\begin{document}

\begin{frontmatter}



\title{Estimations of the discrete Green's function of the  SDFEM  on Shishkin triangular meshes for  singularly perturbed problems with characteristic layers}

\author[label1] {Xiaowei Liu \fnref{cor1}}
\author[label2] {Jin Zhang\corref{cor2}}
\fntext[cor1] {Email: xwliuvivi@hotmail.com }
\cortext[cor2] {Corresponding author: jinzhangalex@hotmail.com }
\address[label1]{College of Science, Qilu University of Technology, Jinan 250353, China}
\address[label2]{School of Mathematics and Statistics, Shandong Normal University,
Jinan 250014, China}

\begin{abstract}
In this technical report, we present estimations of the discrete Green's function of the  streamline diffusion finite element method (SDFEM)  on Shishkin triangular meshes for   singularly perturbed problems with characteristic layers.
\end{abstract}

\end{frontmatter}

\section{Continuous problem, Shishkin mesh, SDFEM}
We consider the singularly perturbed boundary value problem
 \begin{equation}\label{eq:model problem}
 \begin{array}{rcl}
-\varepsilon\Delta u+b  u_x+cu=f & \mbox{in}& \Omega=(0,1)^{2},\\
 u=0 & \mbox{on}& \partial\Omega ,
 \end{array}
 \end{equation}
where $b,c>0$ are constants, $b\ge\beta$ on $\overline{\Omega}$ with a positive constant $\beta$ and $\varepsilon\ll b$ is a small positive parameter. It is assumed that $f$ is sufficiently smooth. The solution of \eqref{eq:model problem} typically has
an exponential layer of width
$O(\varepsilon\ln(1/\varepsilon) )$ near the outflow boundary at $x=1$ and two characteristic (or parabolic) layers of width $O(\sqrt{\varepsilon}\ln(1/\varepsilon))$ near the characteristic boundaries at $y=0$
and $y=1$.

Throughout the article, the standard notation for the Sobolev spaces
and norms will be used; and generic constants $C$, $C_i$ are independent of
$\varepsilon$ and $N$. The constants $C$ are generic while subscripted constants $C_i$ are fixed.

The Shishkin mesh used for discretizing \eqref{eq:model problem} is a piecewise uniform mesh. The reader is referred to \cite{Roos:1998-Layer,Roo1Sty2Tob3:2008-Robust,Linb:2003-Layer} for a detailed discussion of their properties and applications. Mesh changes from coarse to fine are denoted by two mesh  transition parameters $\lambda_x$ and $\lambda_y$. They are defined by   
\begin{equation*}
\lambda_{x}:=\min\left\{ \frac{1}{2},\rho\frac{\varepsilon}{\beta}\ln N \right\} \quad \mbox{and} \quad
\lambda_{y}:=
\min\left\{
\frac{1}{3},\rho\sqrt{\varepsilon}\ln N 
\right\}.
\end{equation*}
where $N=6k$ with $k\in\mathbb{Z}^+$ is the number of mesh intervals in  each direction and $\rho=2.5$ in our analysis for technical reasons   as  in \cite{Zhang:2003-Finite} and \cite{Styn1Tobi2:2003-SDFEM}. Then, the domain $\Omega$ is dissected into four subdomains as $\bar{\Omega}=\Omega_{s}\cup\Omega_{x}\cup\Omega_{y}\cup\Omega_{xy}$(see Fig. \ref{fig:Shishkin mesh}), where
\begin{align*}
&\Omega_{s}:=\left[0,1-\lambda_{x}\right]\times\left[\lambda_{y},1-\lambda_{y}\right],&&
\Omega_{y}:=\left[0,1-\lambda_{x}\right]\times\left(\left[0,\lambda_{y}\right]
\cup\left[1-\lambda_{y},1\right]
\right),\\
&\Omega_{x}:=\left[ 1-\lambda_{x},1 \right]\times\left[\lambda_{y},1-\lambda_{y}\right],&&
\Omega_{xy}:=\left[ 1-\lambda_{x},1 \right]\times\left(\left[0,\lambda_{y}\right]
\cup\left[1-\lambda_{y},1\right]
\right).
\end{align*}

\begin{assumption}\label{assumption: varepsilon-N}
Assume that $\varepsilon\le N^{-1}$, as is generally the case in practice. Furthermore we assume that
$\lambda_{x}=\rho\varepsilon\beta^{-1}\ln N$ and $\lambda_{y}=\rho\sqrt{\varepsilon}\ln N$ as otherwise $N^{-1}$ is exponentially small compared with $\varepsilon$.
\end{assumption}

\begin{figure}
\centering
\includegraphics[width=0.45\columnwidth]{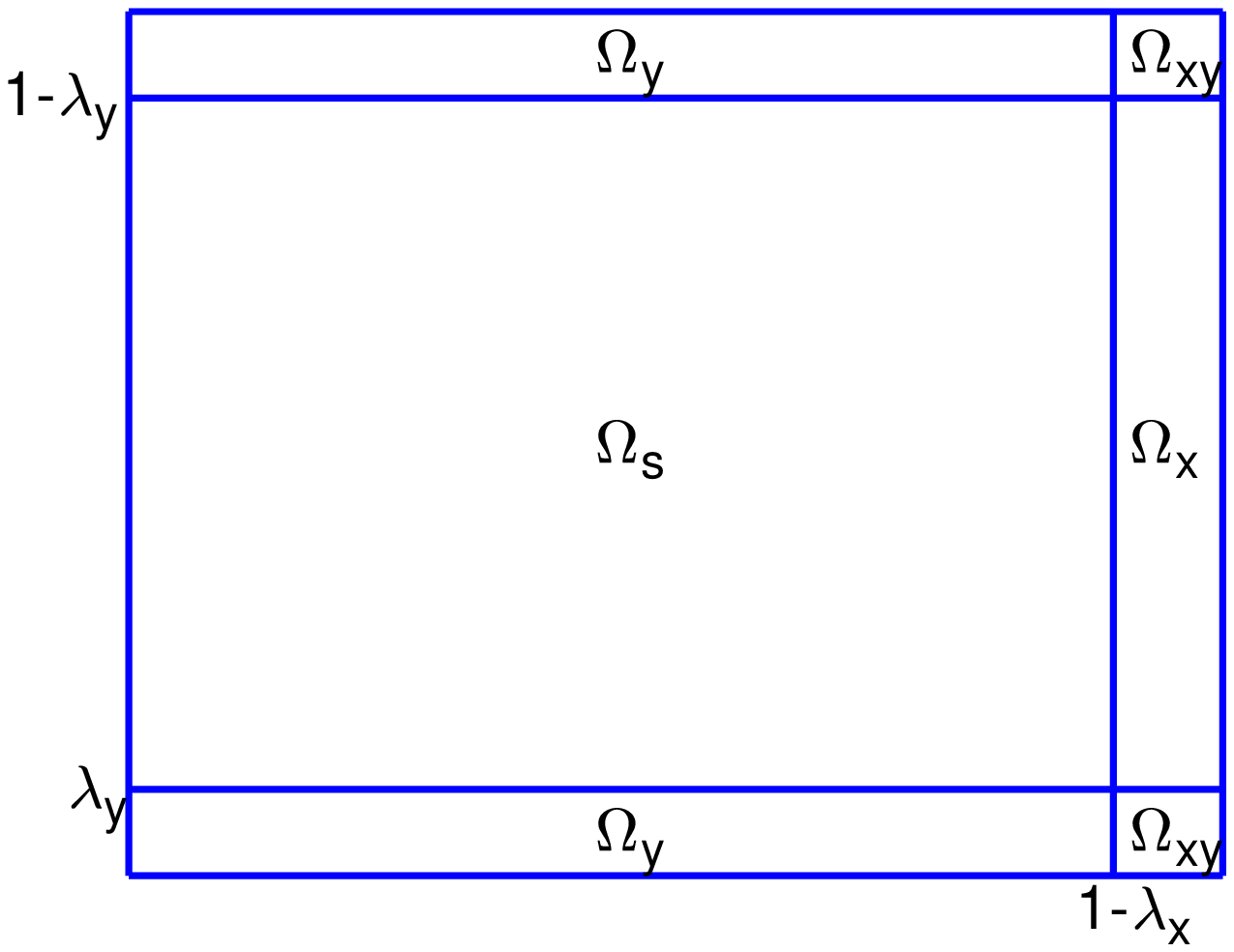}
\includegraphics[width=0.45\columnwidth]{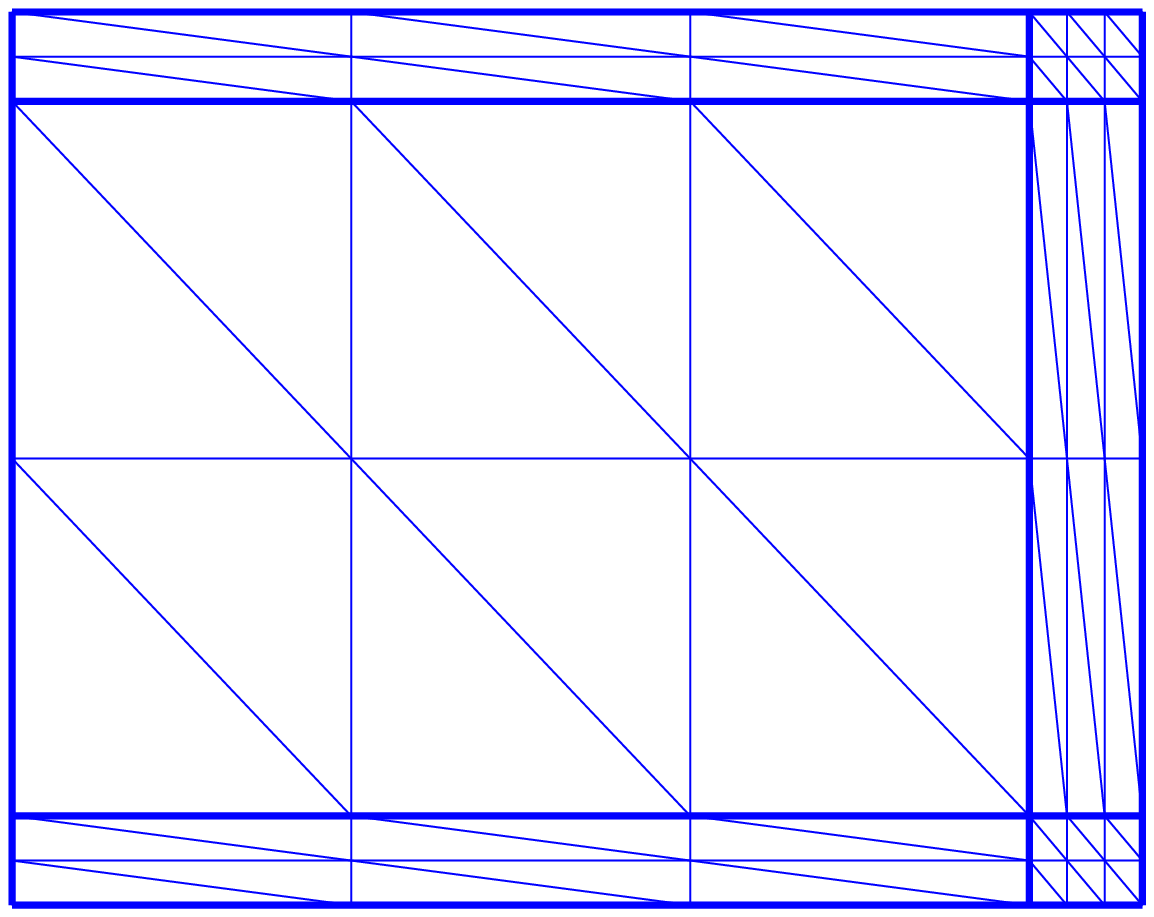}
\caption{\footnotesize{Dissection of $\Omega$ and triangulation $\mathcal{T}_{N}$}}\label{fig:1}
\label{fig:Shishkin mesh}
\end{figure}

\begin{figure}
\centering
\includegraphics[width=0.45\columnwidth]{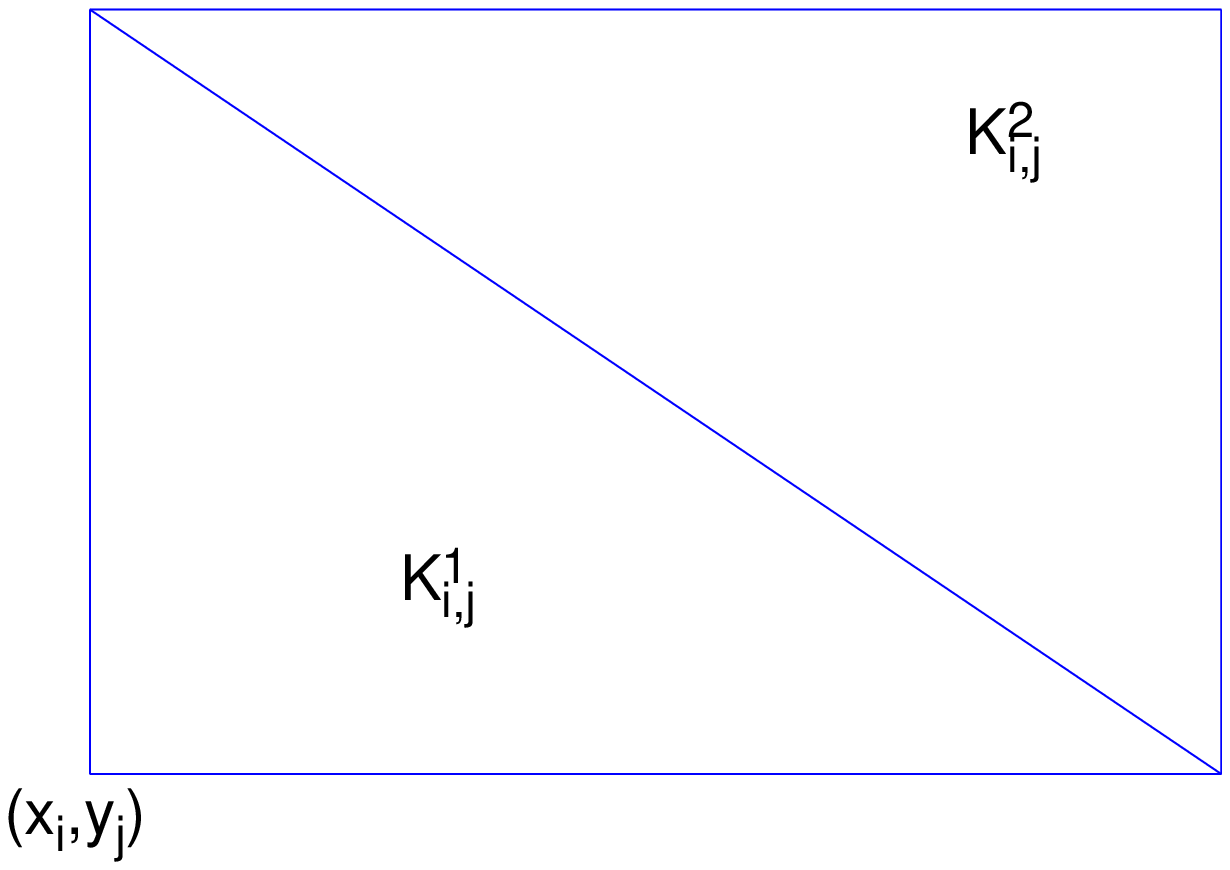}
\caption{$K^{1}_{i,j}$ and $K^{2}_{i,j}$}
\label{fig:code of mesh}
\end{figure}
\par
We introduce the set of mesh points $\left\{ (x_{i},y_{j})\in\Omega:i,\,j=0,\,\cdots,\,N  \right\}$ which divide the subdomains into uniform rectangles and triangles by drawing the diagonal in each rectangle. This triangulation is denoted by $\mathcal{T}_{N}$(see Fig.\ref{fig:Shishkin mesh}). 

The mesh sizes are denoted by $h_{x,i}:=x_{i+1}-x_{i}$ and $h_{y,i}:=y_{i+1}-y_{i}$  which
satisfy
\begin{align*}
N^{-1}\le &h_{x,i}=:H_x,h_{y,j}=:H_y \le 3N^{-1}  \quad 0\le i<N/2,\; N/3\le j<2N/3,\\
C_{1}\varepsilon N^{-1}\ln N &\le h_{x,i}=:h_x\le C_{2}\varepsilon N^{-1}\ln N \quad N/2 \le i<N,\\
C_{1}\sqrt{\varepsilon} N^{-1}\ln N &\le h_{y,j}=:h_y\le C_{2}\sqrt{\varepsilon} N^{-1}\ln N \quad
0\le j<N/3;\; 2N/3\le j<N.
\end{align*}

For notation convenience, we shall use  $K^1_{i,j}$ to denote the triangle with vertices $(x_i,y_j)$, $(x_{i+1},y_j)$ and $(x_i,y_{j+1})$,
$K^2_{i,j}$  for the triangle with vertices $(x_i,y_{j+1})$, $(x_{i+1},y_j)$ and $(x_{i+1},y_{j+1})$ (see  Fig.  \ref{fig:code of mesh}),
and $K$ for a generic element.

Let $V^N\subset V$ be the $C^0$ linear finite element space  on the Shishkin mesh $\mathcal{T}_N$.
The SDFEM consists in adding  weighted residuals to the standard Galerkin method,  which reads
\begin{equation}\label{eq:SDFEM}
\left\{
\begin{array}{lr}
\text{Find $u^{N}\in V^{N}$ such that for all $v^{N}\in V^{N}$},\\
 a_{SD}(u^{N},v^{N})=(f,v^{N})+\underset{K\subset\Omega}\sum(f,\delta_{K}b v^{N}_{x})_{K},
\end{array}
\right.
\end{equation}
where
\begin{equation*}
a_{SD}(u^{N},v^{N})=a_{Gal}(u^{N},v^{N})+a_{stab}(u^{N},v^{N})
\end{equation*}
and
\begin{align*}
a_{stab}(u^{N},v^{N})&=\sum_{K\subset\Omega}(-\varepsilon\Delta u^{N}+bu^{N}_{x}+cu^{N},\delta_{K}bv^{N}_{x})_{K}.
\end{align*}
Note that $\Delta u^N=0$ in $K$ for $u^N\vert_K\in P_1(K)$. Following usual practice
\cite{Roo1Sty2Tob3:2008-Robust},  the stabilization parameter $\delta_{K}=\delta(x,y)|_{K}$ are defined by
\begin{equation*}\label{eq:delta}
\delta(x,y)=
\left\{
\begin{array}{ll}
C^*N^{-1}&\text{if $(x,y)\in\Omega_s\cup\Omega_y$}\\
0&\text{if $(x,y)\in\Omega_x\cup\Omega_{xy}$}
\end{array}
\right.
\end{equation*}
where $C^*$ is a positive constant   independent of 
$\varepsilon$ and the mesh $\mathcal{T}_N$.  The choice of $\delta$ makes the following coercivity hold  
\begin{equation*}\label{eq:SD coercivity}
a_{SD}(v^{N},v^{N})\ge \frac{1}{2}\vvvert v^{N} \vvvert^{2} 
\quad \forall v^{N}\in V^{N}
\end{equation*}
where
\begin{equation*}\label{eq:SD norm}
\vvvert v^{N} \vvvert^{2}:
=\varepsilon \vert v^{N} \vert^{2}_{1}+\Vert v^{N} \Vert^{2}
+\sum_{K\subset\Omega} \delta_{K}\Vert 
bv^{N}_x\Vert^{2}_{K}.
\end{equation*}
Note that  existence and uniqueness of the solution to \eqref{eq:SDFEM} is guaranteed by this coercivity.   Also  Galerkin orthogonality holds, i.e.,
\begin{equation}\label{eq:orthogonality of SD}
a_{SD}(u-u^N,v^N)=0 \quad \forall v^N\in V^N.
\end{equation}

For analysis on Shishkin meshes, we need the following  anisotropic interpolation error bounds given in
\cite[Lemma 3.1]{Linb1Styn2:2001-SDFEM} and \cite[Theorem 1]{Apel1Dobr2:1992-Anisotropic}.
\begin{lemma}\label{lem: anisotropic interpolation-triangle}
Let $K\in\mathcal{T}_N$ and $p\in (1,\infty]$ and
suppose that $K$ is $K^1_{i,j}$ or $K^2_{i,j}$. Assume
that $w\in W^{2,p}(\Omega)$ and denote by $w^I$ the linear function that interpolates to $w$ at the vertices of
$K$.  Then
\begin{align*}
&\Vert w-w^I \Vert_{L^p(K)}\le
C\sum_{l+m=2}h^l_{x,i}h^m_{y,j}\Vert \partial^l_x\partial^m_y w \Vert_{L^p(K)},\\
&\Vert (w-w^I)_x \Vert_{L^p(K)}\le
C\sum_{l+m=1}h^l_{x,i}h^m_{y,j}\Vert \partial^{l+1}_x\partial^m_y w \Vert_{L^p(K)},\\
&\Vert (w-w^I)_y \Vert_{L^p(K)}\le
C\sum_{l+m=1}h^l_{x,i}h^m_{y,j}\Vert \partial^{l}_x\partial^{m+1}_y w \Vert_{L^p(K)}
\end{align*}
where $l$ and $m$ are nonnegative integers.
\end{lemma}
\section{The discrete Green's function}
The estimates of the discrete Green's function
   are similar to ones in \cite{Linb1Styn2:2001-SDFEM} and \cite{Zha1Mei2Che3:2013-Pointwise}. However,  here we present  more \emph{delicate} requirements for the parameters in the weight function, as will be helpful to obtain sharper pointwise bounds.

Let $\boldsymbol{x}^{\ast}=(x^*,y^*)$ be a mesh node in $\Omega$. The discrete Green's function $G\in V^{N}$ associated with $\boldsymbol{x}^{\ast}$ is defined by
\begin{equation}\label{eq:discrete Green function}
a_{SD}(v^{N},G)=v^{N}(\boldsymbol{x}^{\ast})\quad\forall v^{N}\in V^{N}.
\end{equation}
The bound of the discrete Green's function in the energy norm  relies on weight arguments. To start with,  we define a weight function 
\begin{equation*}
\omega(\boldsymbol{x}):=
g\left(\frac{x-x^{\ast}}{\sigma_{x}}\right)
g\left(\frac{y-y^{\ast}}{\sigma_{y}}\right)
g\left(-\frac{y-y^{\ast}}{\sigma_{y}}\right)
\end{equation*}
with $g(r)=2/(1+e^{r})$ for $r\in(-\infty,\infty)$ and
\begin{equation}\label{eq:sigma xy}
\begin{split}
&\sigma_x=k\max\{ N^{-1}, \;\varepsilon\ln^2 N\},\\
&\sigma_y=
\left\{
\begin{array}{ll}
k N^{-1/2} & \text{ if
 $\varepsilon\le N^{-2}$}\\
k \max\{ N^{-3/2} \varepsilon^{-1/2},\;\; \varepsilon^{1/2} \} &\text{ if $N^{-2}\le \varepsilon \le N^{-1}$}
\end{array}
\right.
.
\end{split}
\end{equation}
We shall choose $k>0$ later.

\begin{remark}
The definition \eqref{eq:sigma xy} is different from one in \cite{Zha1Mei2Che3:2013-Pointwise}. It is more delicate and is helpful for our  pointwise estimations. If  $N^{-2}\le \varepsilon \le N^{-1}$, $\sigma_y$  satisfies $k N^{-3/4}\le \sigma_y\le k N^{-1/2}$ and is smaller than one in \cite{Zha1Mei2Che3:2013-Pointwise}.
\end{remark}

\par
For the following analysis, we collect some basic properties of the weight function which can be obtained by some elementary calculations. 
\begin{lemma}\label{elementary properties}
Let $\sigma_x\ge N^{-1}$ and $\sigma_y\ge N^{-1}$. The following estimates hold true for the weight function $\omega(\boldsymbol{x})$:
\begin{enumerate}[(i)]
  \item
  $0<\omega<8$ on $\Omega$;
  \item
  $(\omega^{-1})_{x}>0$ on $\Omega$;
  \item
  for any $K\in\mathcal{T}_{N}$,
$$
\frac{  \underset{K}{\max}\;\omega^{-1}  }{ \underset{K}{\min}\;\omega^{-1} }\le C \;\text{ and } \;\; \frac{  \underset{K}{\max}\; (\omega^{-1})_{x} }{ \underset{K}{\min}\;(\omega^{-1})_{x} }    \le     C;
$$

  \item
  for all $l\ge0$ and $m\ge0$,
   \begin{equation*}
   \left|\frac{\partial^{l+m}\omega(\boldsymbol{x})}{\partial x^{l}\partial y^{m}}\right|
   \le
   C\sigma^{-l}_{x}\sigma^{-m}_{y}\omega(\boldsymbol{x})
   \quad\mbox{on}\quad\Omega;
   \end{equation*}
   \item
    for all $l\ge1$ and $m\ge0$,
    \begin{equation*}
    \left|\frac{\partial^{l+m}\omega(\boldsymbol{x})}{\partial x^{l}\partial y^{m}}\right|
    \le
     C\sigma^{1-l}_{x}\sigma^{-m}_{y}|\omega_{x}(\boldsymbol{x})|
     \quad\mbox{on}\quad\Omega;
    \end{equation*}
    \item
     on any triangle $K^{\ast}$ that contains $\boldsymbol{x}^{\ast}$,
    $ \omega(\boldsymbol{x})\ge C>0$.
\end{enumerate}
\end{lemma}

Now we define a weighted energy norm  
\begin{equation}\label{eq:weighted norm}
\begin{split}
 \vvvert G \vvvert^{2}_{\omega}:= &\varepsilon \Vert \omega^{-1/2}G_x \Vert^{2}
+\varepsilon\Vert \omega^{-1/2}G_y \Vert^{2}+\frac{b}{2}\Vert (\omega^{-1})^{1/2}_x G\Vert^{2}\\
&+c\Vert \omega^{-1/2}G \Vert^{2}
+\sum_K b^{2}\delta_K\Vert \omega^{-1/2}G_x \Vert^{2}_K.
\end{split}
\end{equation}
Note that $(\omega^{-1})_{x}>0$. For any subdomain $D$ of $\Omega$, let
$\vvvert G \vvvert_{\omega,D}$ mean that the integrations in \eqref{eq:weighted norm} are restricted to $D$.
The equalities \eqref{eq:SDFEM}, \eqref{eq:weighted norm} and integration by parts yield
\begin{equation*}\label{eq:norm and bilinear form}
\begin{split}
\vvvert G \vvvert^{2}_{\omega}
=&a_{SD}(\omega^{-1}G,G)-\varepsilon((\omega^{-1})_xG,G_x)
- \varepsilon ((\omega^{-1})_yG,G_y)\\
&
-\sum_K  ( b(\omega^{-1})_xG+c \omega^{-1}G,\delta_K \; bG_x)_K.
\end{split}
\end{equation*}
Considering \eqref{eq:discrete Green function} we  have
\begin{equation*}\label{eq:idea-weighted estimate}
\begin{split}
a_{SD}(\omega^{-1}G,G)&=a_{SD}(\omega^{-1}G-(\omega^{-1}G)^{I},G)+a_{SD}((\omega^{-1}G)^{I},G)\\
&=a_{SD}(\omega^{-1}G-(\omega^{-1}G)^{I},G)+(\omega^{-1}G)(\boldsymbol{x}^{\ast}).
\end{split}
\end{equation*}
With the above two equalities, the weighted energy estimate of $G$ will be obtained by means of  the next three Lemmas.
\begin{lemma}\label{lem:1}
If $\sigma_x\ge kN^{-1} $ and $\sigma_y\ge k \varepsilon^{1/2}$  for $k>1$ sufficiently
large and independent of $N$ and $\varepsilon$, we have
\begin{equation*}
a_{SD}(\omega^{-1}G,G)\ge \frac{1}{4}\vvvert G \vvvert^{2}_{\omega}.
\end{equation*}
\end{lemma}
\begin{proof}
See \cite[Lemma 4.2]{Linb1Styn2:2001-SDFEM}.
\end{proof}

\begin{lemma}\label{lem:2}
Assume $\sigma_x\ge kN^{-1}$ with $k>0$ independent of $N$ and $\varepsilon$. Then for each mesh point $\boldsymbol{x}^{\ast}\in\Omega_{s}\cup\Omega_{x}$, we have
\begin{equation*}
\left|(\omega^{-1}G)(\boldsymbol{x}^{\ast})\right|\le
\frac{1}{16}\vvvert G \vvvert^{2}_{\omega}
+
\left\{
\begin{matrix}
CN^2\sigma_x& \text{if $\b{x}^*\in\Omega_s$}\\
CN\ln N& \text{if $\b{x}^*\in\Omega_x$}
\end{matrix}
\right.
,
\end{equation*}
where $C$ is independent of $N$, $\varepsilon$ and $\boldsymbol{x}^{\ast}$.
\end{lemma}
\begin{proof}
See \cite[Lemma 4.3]{Linb1Styn2:2001-SDFEM}.
\end{proof}

\begin{lemma}\label{lem:3}
If $\sigma_{x}$ and $\sigma_{y}$ satisfy \eqref{eq:sigma xy}, where $k>1$ is sufficiently
large and independent of $N$ and $\varepsilon$, then
\begin{equation*}
a_{SD}((\omega^{-1}G)^{I}-\omega^{-1}G,G)\le \frac{1}{16}\vvvert G \vvvert^{2}_{\omega}.
\end{equation*}
\end{lemma}
\begin{proof}

For convenience we set
$\tilde{E}(\boldsymbol{x}):=((\omega^{-1}G)^{I}-\omega^{-1}G)(\boldsymbol{x})$.
Integration by parts yields $(b\tilde{E}_{x},G)=-(b\tilde{E},G_{x})$,  and Cauchy-Schwarz inequalities yield
\begin{equation}\label{B(E,G)}
\begin{split}
|a_{SD}(\tilde{E},G)|
\le
C\big( &\Vert (\varepsilon+b^{2}\delta)^{1/2}\omega^{1/2}\tilde{E}_{x} \Vert
+\varepsilon^{1/2}\Vert \omega^{1/2}\tilde{E}_{y} \Vert\\
&+\Vert (\varepsilon+b^{2}\delta)^{-1/2}\omega^{1/2}\tilde{E} \Vert
\big) \vvvert  G \vvvert_{\omega}.
\end{split}
\end{equation}

\noindent
\textbf{Step 1}. To analyze different kinds of interpolation bounds, we first estimate the derivatives of the weighted discrete Green's functions. Note that $G_{xx}=G_{yy}=G_{xy}=0$ on $K$ 
because of  $G|_K\in P_1(K)$. 
For convenience, we set $M_{K}:=\underset{K}{\max}\,\omega^{-1/2}$.

Using Lemma \ref{elementary properties} (iii)--(v), we obtain
\begin{equation}\label{eq:omega-G-xx}
\begin{split}
&\Vert (\omega^{-1}G)_{xx} \Vert_{K}
\le
\Vert(\omega^{-1})_{xx}G \Vert_{K}+2\Vert(\omega^{-1})_{x}G_{x}\Vert_{K}\\
\le&
CM_{K}
\left(  
\sigma^{-3/2}_{x}\Vert (\omega^{-1})^{1/2}_{x}G \Vert_{K}+
\sigma^{-1}_{x}\Vert \omega^{-1/2} G_{x} \Vert_{K}
\right)\\
\le&
CM_{K}
\left( 
\sigma^{-3/2}_{x}+
\sigma^{-1}_{x}(\varepsilon+b^{2}\delta)^{-1/2}
\right)
\vvvert G \vvvert_{\omega,K}.
\end{split}
\end{equation}

Note
$\Vert     G_{ y} \Vert_{K}\le C h^{-1}_{y,K} \Vert     G  \Vert_{K}$
or
$\Vert     G_{ y} \Vert_{K}\le C \varepsilon^{-1/2}\cdot \varepsilon^{1/2}\Vert     G_y  \Vert_{K}$,
then one has
$$
\Vert     G_{ y} \Vert_{K}\le C\min\{ h^{-1}_{y,K}, \varepsilon^{-1/2} \} 
\vvvert G \vvvert_{K},
$$
and
\begin{equation*}\label{eq:omega-y-G-y}
\begin{split}
&\Vert  (\omega^{-1})_{ y}G_{ y} \Vert_{K} 
\le 
C\underset{K}{\max} |(\omega^{-1})_y|\;   \Vert     G_{ y} \Vert_{K} \\
\le &
CM_{K}\sigma^{-1}_{ y}  \underset{K}{\max}\omega^{-1/2} \cdot 
\min\{ h^{-1}_{y,K}, \varepsilon^{-1/2} \} 
\vvvert G \vvvert_{K}\\
\le &
C M_{K} \sigma^{-1}_{ y} \min\{ h^{-1}_{y,K}, \varepsilon^{-1/2} \} 
\vvvert G \vvvert_{\omega,K},
\end{split}
\end{equation*}
where we have used (iii) in Lemma \ref{elementary properties}, for example
$$
\underset{K}{\max}\omega^{-1/2} \cdot \Vert    G \Vert_{K}
\le 
\frac{ \underset{K}{\max}\omega^{-1/2} }{ \underset{K}{\min}\omega^{-1/2} }
 \; \underset{K}{\min}\omega^{-1/2} \Vert    G \Vert_{K}
 \le 
 C \Vert   \omega^{-1/2} G \Vert_{K}.
$$
Similarly,  we have $\Vert(\omega^{-1})_{yy}G \Vert_{K}\le CM_{K} \sigma^{-2}_{y}\Vert \omega^{-1/2}G \Vert_{K}$
and 
\begin{equation}\label{eq:omega-G-yy}
\begin{split}
&\Vert (\omega^{-1}G)_{yy} \Vert_{K}
\le
\Vert(\omega^{-1})_{yy}G \Vert_{K}+\Vert(\omega^{-1})_{y}G_{y}\Vert_{K}\\
\le &
CM_{K}
\left( \sigma^{-2}_{y}+ \sigma^{-1}_{ y} \min\{ h^{-1}_{y,K}, \varepsilon^{-1/2} \}  \right)
\vvvert  G \vvvert_{\omega,K}.
\end{split}
\end{equation}

Recalling (iii) in Lemma \ref{elementary properties} and inverse estimates  \cite[Theorem 3.2.6]{Ciarlet:2002-finite} , we have
\begin{equation}\label{eq:omega-x-G-y-I}
\begin{split}
\Vert  (\omega^{-1})_{ x }G_{  y}  \Vert_{K}
&\le
C \underset{K}{\max} (\omega^{-1})_{ x }\cdot  \Vert G_{  y}  \Vert_{K}
\le
C \underset{K}{\max} (\omega^{-1})_{ x }\cdot h^{-1}_{y,K}   \Vert G  \Vert_{K} \\
&\le
C h^{-1}_{y,K}\left( \underset{K}{\max} (\omega^{-1})_{ x } \right)^{1/2}
\left( \underset{K}{\min} (\omega^{-1})_{ x } \right)^{1/2}
\cdot     \Vert G  \Vert_{K}\\
 &\le
CM_{K} \cdot h^{-1}_{y,K} \sigma^{-1/2}_{ x}\cdot     \Vert (\omega^{-1})^{1/2}_{ x } G  \Vert_{K}.
\end{split}
\end{equation}
Also, we have
\begin{equation}\label{eq:omega-x-G-y-II}
\Vert  (\omega^{-1})_{ x }G_{  y}  \Vert_{K}
\le
CM_{K}\cdot
 \varepsilon^{-1/2}\sigma^{-1}_{ x}\cdot  \varepsilon^{1/2}\Vert  \omega^{-1/2}G_{ y} \Vert_{K}.
\end{equation}
Then from \eqref{eq:omega-x-G-y-I} and \eqref{eq:omega-x-G-y-II}, one has
\begin{equation*}
\Vert  (\omega^{-1})_{ x }G_{  y}  \Vert_{K}
\le 
C M_{K} \min\{ h^{-1}_{y,K} \sigma^{-1/2}_{ x}  ,\varepsilon^{-1/2}\sigma^{-1}_{ x}\} \vvvert G \vvvert_{\omega,K},
\end{equation*}
and then
\begin{equation}\label{eq:omega-G-xy}
\begin{split}
&\Vert (\omega^{-1}G)_{xy} \Vert_{K}
\le
\Vert(\omega^{-1})_{xy}G \Vert_{K}+\Vert(\omega^{-1})_{y}G_{x}\Vert_{K}
+\Vert(\omega^{-1})_{x}G_{y}\Vert_{K}\\
\le &
CM_{K}
\left( \sigma^{-1/2}_{x}\sigma^{-1}_{y}\Vert (\omega^{-1})^{1/2}_{x}G \Vert_{K}+
\sigma^{-1}_{y}\Vert \omega^{-1/2} G_{x} \Vert_{K}
\right)+\Vert(\omega^{-1})_{x}G_{y}\Vert_{K}\\
\le &
CM_{K}
\big(  \sigma^{-1/2}_{x}\sigma^{-1}_{y}+
\sigma^{-1}_y(\varepsilon+b^{2}\delta)^{-1/2}\\
&+\min\{ h^{-1}_{y,K} \sigma^{-1/2}_{ x}  ,\varepsilon^{-1/2}\sigma^{-1}_{ x}\} \big)
\vvvert  G \vvvert_{\omega,K}.
\end{split}
\end{equation}

\noindent
\textbf{Step 2}.  Now, we will analyze $\nabla\tilde{E}$ and $\tilde{E}$ respectively. \\
\noindent
(a) From Lemma \ref{lem: anisotropic interpolation-triangle}, we obtain
\begin{align}
\Vert \tilde{E}_{x} \Vert_{K}
&\le
C\big(h_{x,K}\Vert(\omega^{-1}G)_{xx}\Vert_{K}+
h_{y,K}\Vert(\omega^{-1}G)_{xy}\Vert_{K}\big),\label{interpolation of Ex}\\
\Vert \tilde{E}_{y} \Vert_{K}
&\le
C
\big(h_{x,K}\Vert(\omega^{-1}G)_{xy}\Vert_{K}+
h_{y,K}\Vert(\omega^{-1}G)_{yy}\Vert_{K}\big).\label{interpolation of Ey}
\end{align}
Substituting \eqref{eq:omega-G-xx}, \eqref{eq:omega-G-yy} and \eqref{eq:omega-G-xy} into \eqref{interpolation of Ex} and \eqref{interpolation of Ey},  we have
\begin{align}
(\varepsilon+b^{2}\delta)\Vert \omega^{1/2}\tilde{E}_{x} \Vert^{2}
&\le
Ck^{-2}\vvvert G \vvvert^{2}_{\omega},\label{Ex}\\
\varepsilon\Vert\omega^{1/2} \tilde{E}_{y} \Vert^{2}
&\le
Ck^{-2}\vvvert  G \vvvert^{2}_{\omega}\label{Ey}.
\end{align}
More precisely, we have
\begin{equation}\label{eq:Ex-Omega-x}
\Vert \omega^{1/2}\tilde{E}_{x} \Vert_{ \Omega_x }
\le
Ck^{-1}( \varepsilon^{-1/2} \ln^{-1}N +\varepsilon^{-1/2}N^{-1}\sigma^{-1}_y)\vvvert  G \vvvert_{\omega},
\end{equation}
where we have used $ \sigma_x\ge  k\varepsilon\ln^2 N$.\\
(b) Lemma \ref{lem: anisotropic interpolation-triangle} yields
\begin{align*}
\Vert \tilde{E} \Vert_{K}
\le&
C\left(
h^{2}_{x,K}\Vert(\omega^{-1}G)_{xx}\Vert_{K}
+h_{x,K}h_{y,K}\Vert(\omega^{-1}G)_{xy}\Vert_{K}
+h^{2}_{y,K}\Vert(\omega^{-1}G)_{yy}\Vert_{K}
\right).
\end{align*}
Substituting \eqref{eq:omega-G-xx}, \eqref{eq:omega-G-yy} and \eqref{eq:omega-G-xy} into the above inequality, 
for $K\subset\Omega_s\cup\Omega_x\cup\Omega_y$ we have
\begin{equation}\label{E Omegas Omegay}
\Vert \omega^{1/2}\tilde{E}\Vert_{K}
\le
Ck^{-1}N^{-1/2}\vvvert  G \vvvert_{\omega,K},
\end{equation}
and for $K\subset\Omega_{xy}$ 
\begin{equation}\label{E Omega-xy}
\Vert \omega^{1/2}\tilde{E}\Vert_{K}
\le
Ck^{-1}\varepsilon^{1/2}\vvvert  G \vvvert_{\omega,K}.
\end{equation}

For what follows we need a sharper bound of $\Vert \omega^{1/2}\tilde{E} \Vert_{\Omega_x }$.    Similar to \cite[Lemma 4.4]{Linb1Styn2:2001-SDFEM}   we  consider \eqref{eq:sigma xy}, \eqref{E Omegas Omegay} and  \eqref{eq:Ex-Omega-x} and obtain 
\begin{equation}\label{E Omegax}
\begin{split}
&\Vert \omega^{1/2}\tilde{E} \Vert^{2}_{\Omega_x }
\le
C\lambda^{2}_{x}
\left\{\Vert(\omega^{1/2})_{x}\tilde{E}\Vert^{2}_{\Omega_x }+\Vert\omega^{1/2}\tilde{E}_{x}\Vert^{2}_{\Omega_x }\right\}\\
\le &
C\varepsilon^{2}\ln^{2}N\cdot
\left\{\sigma^{-2}_{x}\Vert \omega^{1/2}\tilde{E} \Vert^{2}_{\Omega_x }+\Vert\omega^{1/2}\tilde{E}_{x}\Vert^{2}_{\Omega_x }\right\}\\
\le  &
Ck^{-2}\varepsilon^{2}\ln^{2}N
    \{\sigma^{-2}_{x} N^{-1} +\varepsilon^{-1} \ln^{-2}N +\varepsilon^{-1 }N^{-2}\sigma^{-2}_y
    \}\vvvert G \vvvert^{2}_{\omega}\\
\le     &
    Ck^{-2}\varepsilon  \vvvert G \vvvert^{2}_{\omega},
    \end{split}
\end{equation}
where we have  used $N^{-1} \ln^4 N\le C $ for $N\ge 2$.
\par
Substituting \eqref{Ex}, \eqref{Ey} and \eqref{E Omegas Omegay}--\eqref{E Omegax} into \eqref{B(E,G)} and recalling the definition
of $\delta$, we obtain
\begin{equation*}
|a_{SD}(\tilde{E},G)|\le
Ck^{-1}\vvvert G \vvvert^{2}_{\omega}.
\end{equation*}
Choosing $k$ sufficiently large independently
of $\varepsilon$ and $N$, we are done.
\end{proof}

Lemmas 
\ref{lem:1}, \ref{lem:2} and \ref{lem:3}
 yield the following bound of the discrete Green function in the energy norm.
\begin{theorem}\label{theorem energy Green}
Assume  that $\sigma_{x}$ and $\sigma_{y}$ satisfy \eqref{eq:sigma xy},  where
$k$ is chosen so  that Lemmas \ref{lem:1}, \ref{lem:2} and \ref{lem:3} hold. 
For $\boldsymbol{x}^{*}\in \Omega_s\cup\Omega_{x}$ we have
\begin{equation*}
\vvvert G \vvvert^2 \le  8\vvvert G \vvvert^2_{\omega} \le 
\left\{
\begin{matrix}
CN^2 \sigma_x& \text{if $\b{x}^*\in\Omega_s$}\\
CN \ln  N& \text{if $\b{x}^*\in\Omega_x$}
\end{matrix}
\right..
\end{equation*}
\end {theorem}
\begin{proof}
See \cite[Theorem 4.1]{Linb1Styn2:2001-SDFEM}.
\end{proof}


%




\begin{thebibliography}{1}

\bibitem{Apel1Dobr2:1992-Anisotropic}
T.~Apel and M.~Dobrowolski.
\newblock Anisotropic interpolation with applications to the finite element
  method.
\newblock {\em Computing}, 47(3):277--293, 1992.

\bibitem{Ciarlet:2002-finite}
P.~G. Ciarlet.
\newblock {\em The Finite Element Method for Elliptic Problems}, volume~40 of
  {\em Classics in Applied Mathematics}.
\newblock Society for Industrial and Applied Mathematics (SIAM), Philadelphia,
  PA, 2002.

\bibitem{Linb:2003-Layer}
T.~Lin\ss.
\newblock {\em Layer-adapted meshes for reaction-convection-diffusion
  problems}, volume 1985 of {\em Lecture Notes in Mathematics}.
\newblock Springer-Verlag, Berlin, 2010.

\bibitem{Linb1Styn2:2001-SDFEM}
T.~Lin{\ss} and M.~Stynes.
\newblock The {SDFEM} on {S}hishkin meshes for linear convection--diffusion
  problems.
\newblock {\em Numer. Math.}, 87(3):457--484, 2001.

\bibitem{Roos:1998-Layer}
H.{-G}. Roos.
\newblock Layer-adapted grids for singular perturbation problems.
\newblock {\em ZAMM Z. Angew. Math. Mech.}, 78(5):291--309, 1998.

\bibitem{Roo1Sty2Tob3:2008-Robust}
H.{-G}. Roos, M.~Stynes, and L.~Tobiska.
\newblock {\em Robust Numerical Methods for Singularly Perturbed Differential
  Equations}, volume~24 of {\em Springer Series in Computational Mathematics}.
\newblock Springer-Verlag, Berlin, second edition, 2008.

\bibitem{Styn1Tobi2:2003-SDFEM}
M.~Stynes and L.~Tobiska.
\newblock The {SDFEM} for a convection-diffusion problem with a boundary layer:
  optimal error analysis and enhancement of accuracy.
\newblock {\em SIAM J. Numer. Anal.}, 41(5):1620--1642, 2003.

\bibitem{Zha1Mei2Che3:2013-Pointwise}
J.~Zhang, L.~Mei, and Y.~Chen.
\newblock Pointwise estimates of the {SDFEM} for convection-diffusion problems
  with characteristic layers.
\newblock {\em Appl. Numer. Math.}, 64:19--34, 2013.

\bibitem{Zhang:2003-Finite}
Z.~Zhang.
\newblock Finite element superconvergence on {S}hishkin mesh for 2-{D}
  convection-diffusion problems.
\newblock {\em Math. Comp.}, 72(243):1147--1177, 2003.

\end{thebibliography}







\end{document}